\documentclass{amsart}
\calclayout
 \usepackage{amssymb,amsmath,amsfonts,epsfig,latexsym,tikz}
   \usepackage[alphabetic]{amsrefs}
 \usepackage{tikz-cd}
\usepackage{hyperref}
\usepackage{enumerate}
\usepackage{mathtools}
\usepackage{verbatim}
\usepackage{cleveref}
\usepackage[shortlabels]{enumitem}
\usepackage{subcaption}

\usetikzlibrary{positioning}
\usetikzlibrary{matrix}
\usetikzlibrary{decorations}
\usetikzlibrary{decorations.pathreplacing, decorations.pathmorphing, angles,quotes}
 
 \newtheorem{theorem}{Theorem}[section]

\newtheorem{proposition}[theorem]{Proposition}
\newtheorem{lemma}[theorem]{Lemma}
\newtheorem{corollary}[theorem]{Corollary}

\theoremstyle{definition}
\newtheorem{remark}[theorem]{Remark}
\newtheorem{example}[theorem]{Example}

\begin{document}

\title{$F$-rationality of multiplicity-free subvarieties}          
    
\author{Matt Larson}
\date{\today}
\begin{abstract}
A multiplicity-free subvariety of a product of projective lines is a subvariety with the property that the expansion of its Chow class in the usual basis has all coefficients equal to $0$ or $1$. We show that, over a field of positive characteristic, a multiplicity-free subvariety has $F$-rational singularities. This strengthens Brion's result that multiplicity-free subvarieties are normal and Cohen--Macaulay, and that they have rational singularities over a field of characteristic $0$. 
\end{abstract}
 
\maketitle

\vspace{-20 pt}

\section{Introduction}

Let $k$ be a field, and let $Y$ be an (integral) subvariety of $(\mathbb{P}^1)^n$ over $k$. There is a canonical basis for the Chow homology of $(\mathbb{P}^1)^n$ given by the classes of coordinate subspaces. We say that $Y$ is \emph{multiplicity-free} if the expansion of the fundamental class of $Y$ in this basis has all coefficients equal to $0$ or $1$. 
There are a number of natural examples of multiplicity-free subvarieties; see Section~\ref{sec:examples}.

By Bertini's theorem, whether $Y$ is multiplicity-free is determined by the intersection of $Y$ with a dense open subset of $(\mathbb{P}^1)^n$. This makes it perhaps surprising that the singularities of multiplicity-free varieties are mild, as the following result of Brion shows. 

\begin{proposition}\cites{BrionOrbitClosure,BrionMultiplicity}\label{prop:CM}
Let $Y$ be a multiplicity-free integral subvariety of $(\mathbb{P}^1)^n$ over a field $k$. Then $Y$ is normal and Cohen--Macaulay. If $k$ has characteristic $0$, then $Y$ has rational singularities. 
\end{proposition}

This result has played a crucial role in work of Berget and Fink \cites{BergetFink,BF24} and Liu \cite{LiuVanishing}. Related results also played an important role in \cite{EFL}. 
A different proof of the normality and Cohen--Macaulayness was given in \cite[Section 6]{CCRC}. Proposition~\ref{prop:CM} can be used to recover several results from the literature \cite{BWMW} \cite[Corollary 2.3]{Sarkar} \cite[Corollary 4.1]{BM}; see Section~\ref{sec:examples}. 

We study the $F$-singularities of multiplicity-free varieties over a field of positive characteristic. These are classes of singularities which are defined in terms of the Frobenius morphism. We will be particularly concerned with \emph{$F$-rational singularities}. This class of singularities was originally defined in terms of tight closure \cites{HH90,HH94,FedderWatanabe}, and it can also be defined in terms of the action of Frobenius on local cohomology \cite{Smith94}.

We will use a characterization of $F$-rationality given in \cite[Corollary 3.6]{BST}. An (integral) variety $Y$ over an algebraically closed field has $F$-rational singularities if it is Cohen--Macaulay, and, for any alteration $\pi \colon X \to Y$, the pushforward map $\pi_*\omega_X \to \omega_Y$ is surjective. Here $\omega_X$ and $\omega_Y$ are the canonical sheaves of $X$ and $Y$. We say that a variety $Y$ over a field $k$ has geometrically $F$-rational singularities if its base change to an algebraically closed field has $F$-rational singularities; this implies that $Y$ has $F$-rational singularities in the usual sense. See \cite[Proposition A.5]{DattMurayama} and \cite[(6) on pg. 440]{Velez}. 

\begin{theorem}\label{thm:main}
Let $Y$ be a multiplicity-free integral subvariety of $(\mathbb{P}^1)^n$ over a field $k$ of positive characteristic. Then $Y$ has geometrically $F$-rational singularities. 
\end{theorem}

If $Y$ is a variety over a field of characteristic $0$, then it is known that $Y$ has rational singularities if and only if, when $Y$ is spread out to mixed characteristic, there is a dense subset of the base over which the fibers have $F$-rational singularities \cite{SmithRational,Hara}.  In particular, Theorem~\ref{thm:main} is a strengthening of Proposition~\ref{prop:CM}, as we can spread out $Y$ together with its realization as a multiplicity-free subvariety. In Section~\ref{sec:examples}, we give examples which show that multiplicity-free varieties need not have strongly $F$-regular or even $F$-pure singularities. 

As we describe in Section~\ref{sec:examples}, Theorem~\ref{thm:main} can be used to recover results of \cite{BW}, \cite{Fsingsquarefree}, and \cite{BDSW}. 

The proof of Theorem~\ref{thm:main} is somewhat similar to Brion's proof of Proposition~\ref{prop:CM}, see Remark~\ref{rem:rational}, although there are significant differences. One key tool that we use is a vanishing theorem  due to Blickle, Schwede, and Tucker \cite[Section 5]{BST}, which was inspired by a result of Bhatt \cite[Theorem 0.5]{Bhatt}.

\medskip

In \cites{BrionOrbitClosure,BrionMultiplicity}, Brion works in a slightly more general setting: he considers a (generalized) flag variety, given as the quotient of a reductive group $G$ by a reduced parabolic $P$. Then the Chow homology of $G/P$ is equipped with a Schubert basis, and he considers subvarieties $Y$ of $G/P$ with the property that the expansion of the fundamental class of $Y$ in this basis has all coefficients equal to $0$ or $1$. If $G = (GL_2)^n$ and $P$ is a Borel subgroup, then $G/P = (\mathbb{P}^1)^n$ and this is the notion described above. 
Brion proves Proposition~\ref{prop:CM} in this level of generality, and our proof of Theorem~\ref{thm:main} works in this level of generality. See Section~\ref{ssec:GmodP}. 

We choose to focus on the case of $(\mathbb{P}^1)^n$ for clarity of presentation, and also because that is where most of the applications are. As shown in Proposition~\ref{prop:product}, a multiplicity-free subvariety of a product of projective spaces can be covered by open subsets which are isomorphic to open subsets of multiplicity-free subvarieties of a product of projective lines, so Theorem~\ref{thm:main} applies directly to this case. For more general $G/P$, there are not very many interesting examples of multiplicity-free subvarieties. The most obvious examples are Schubert varieties, but it is known that these satisfy the stronger property of being strongly $F$-regular \cite{SchubertFregular}. In \cite{BrionOrbitClosure}, Brion shows that, for certain subgroups $H$ of $G$, $H$-orbit closures in $G/P$ are multiplicity-free; see also \cites{Ressayre,BenderPerrin}. In \cites{HeThomsen}, He and Thomsen show that these $H$-orbit closures are often strongly $F$-regular.

\subsection*{Acknowledgements}
We thank Karen Smith for asking us about $F$-singularities of multiplicity-free varieties. 
We thank Bhargav Bhatt and Linquan Ma for answering our questions about $F$-singularities. We thank Colin Crowley, Anna Brosowsky, and Connor Simpson for helpful discussions, especially about Example~\ref{ex:schubert}. 
Claude Opus 5 was used for proofreading, but AI had no role in the content or writing of this paper. 
This work was conducted while the author was at the Institute for Advanced Study, where he is supported by the Charles Simonyi Endowment and the Oswald Veblen Fund.

\section{Examples of multiplicity-free varieties}\label{sec:examples}

In this section, we give several examples of multiplicity-free subvarieties of a product of projective lines. We begin by proving some basic results about multiplicity-free varieties.

\subsection{Tools}

The first three results in this section are needed in the proof of Theorem~\ref{thm:main}.

\begin{proposition}\label{prop:projectionproduct}
Let $Y$ be a subvariety of $(\mathbb{P}^1)^n$, and let $p \colon (\mathbb{P}^1)^n \to (\mathbb{P}^1)^{n-1}$ be a coordinate projection. If the restriction of $p$ to $Y$ is not generically finite, then the natural map from $Y$ to $p(Y) \times \mathbb{P}^1$ is an isomorphism. 
\end{proposition}

Here the map $Y \to \mathbb{P}^1$ is the map to the factor which is forgotten by $p$. 

\begin{proof}
Note that $Y$ is contained in $p^{-1}(p(Y))$. Furthermore, the dimension of $p(Y)$ is $\dim Y - 1$ because the restriction of $p$ to $Y$ is not generically finite, so $Y$ has the same dimension as $p^{-1}(p(Y))$. The result follows because $Y$ is integral. 
\end{proof}

\begin{proposition}\label{prop:image}
Let $Y$ be a multiplicity-free subvariety of $(\mathbb{P}^1)^n$. Then the image of $Y$ under a coordinate projection is multiplicity-free. 
\end{proposition}

\begin{proof}
It suffices to consider the case of a coordinate projection $p$ which forgets one factor, i.e., $p$ is a map to $(\mathbb{P}^1)^{n-1}$. If $p$ is not generically finite, then the result follows from Proposition~\ref{prop:projectionproduct}. If $p$ is generically finite, then $d[p(Y)] = p_*[Y]$ in the Chow homology of $(\mathbb{P}^1)^{n-1}$, where $d$ is the degree of the restriction of $p$ to $Y$. As $p_*[Y]$ has all coefficients equal to $0$ or $1$, we see that $d=1$ and that $[p(Y)]$ has all coefficients equal to $0$ or $1$. 
\end{proof}

The proof of Proposition~\ref{prop:image} shows that if the restriction of a coordinate projection to $Y$ is generically finite, then it is birational. 

\begin{proposition}\label{prop:geomintegral}
Let $Y$ be a multiplicity-free subvariety of $(\mathbb{P}^1)^n$. Then $Y$ is geometrically integral. 
\end{proposition}

\begin{proof}
Let $K$ be an extension of $k$, and let $Y_K$ be the base change of $Y$ to $K$. The classes of $Y_K$ and $Y$ in the Chow homology of $(\mathbb{P}^1)^n$ have the same expression in the bases of classes of coordinate subspaces. 
There is some finite Galois extension of $k$ such that the components of $Y_K$ are defined over that extension, and the Galois group of that extension acts transitively on the components \cite[Section 0364]{stacks-project}. In particular, the fundamental class of each component is the same. The multiplicity-free property then implies that $Y_K$ is irreducible. 

By \cite[04KS]{stacks-project}, if $Y_K$ is non-reduced, then $Y_K$ is generically non-reduced. This contradicts the multiplicity-free property. 
\end{proof}

\medskip

The remaining results will be used to give examples. 
First, we mention three other descriptions of multiplicity-free varieties. In \cites{CS2,CSSurvey}, Conca, de Negri, and Gorla study a family of multihomogeneous ideals in a multigraded polynomial ring that they call \emph{Cartwright--Sturmfels ideals}. These are multihomogeneous ideals whose generic initial ideal is radical. The saturated ideals in the coordinate ring $k[x_1, \dotsc, x_n, y_1, \dotsc, y_n]$ of $\mathbb{A}^{2n}$ which correspond to multiplicity-free subvarieties of $(\mathbb{P}^1)^n$ are exactly prime Cartwright--Sturmfels ideals. 

In \cite[Remark 2.15]{HL24}, it is shown that a subvariety $X$ of $\mathbb{A}^n$ has the property that the closure of $X$ in $(\mathbb{P}^1)^n$ is multiplicity-free if and only if all Gr\"{o}bner degenerations of $X$ are generically reduced. This furthermore implies that all Gr\"{o}bner degenerations of $X$ are reduced, and that the ideal of $X$ has a universal Gr\"{o}bner basis consisting of squarefree-supported polynomials, i.e., all monomials appearing in each polynomial with nonzero coefficient are squarefree. 

Finally, multiplicity-free subvarieties of $(\mathbb{P}^1)^n$ are exactly integral subvarieties of $(\mathbb{P}^1)^n$ that are \emph{kindred} in the sense of \cite[Definition 2.1]{EFL}, which is a condition on the class of the structure sheaf in $K((\mathbb{P}^1)^n)$. See \cite[Proposition 2.11]{EFL}, which is a consequence of the main result of \cite{BrionMultiplicity}.

\begin{proposition}\label{prop:slice}
Let $Y$ be a multiplicity-free subvariety of $(\mathbb{P}^1)^n$, and let $p \colon (\mathbb{P}^1)^n \to (\mathbb{P}^1)^{\ell}$ be a coordinate projection. Then the fibers of $p$ are reduced, and any irreducible component of any fiber of $p$ is multiplicity-free. 
\end{proposition}

\begin{proof}
Choose coordinates $x_1, \dotsc, x_n$ on $(\mathbb{P}^1)^n$, and assume that the map to $(\mathbb{P}^1)^{\ell}$ is the projection onto the first $\ell$ coordinates. It suffices to consider the fiber over the origin, i.e., to show that $Y \cap V(x_1, \dotsc, x_{\ell})$ is reduced and that all irreducible components of this fiber are multiplicity-free subvarieties. 

To prove that $Y \cap V(x_1, \dotsc, x_{\ell})$ is reduced, we can work locally. View $\mathbb{A}^n$ as embedded in $(\mathbb{P}^1)^n$ in the usual way. It suffices to show that $Y \cap \mathbb{A}^n \cap V(x_1, \dotsc, x_{\ell})$ is reduced. By \cite[Remark 2.15]{HL24}, the affine variety $Y \cap \mathbb{A}^n$ has a universal Gr\"{o}bner basis consisting of squarefree-supported polynomials. By \cite[Example 2.16]{HL24}, $Y \cap \mathbb{A}^n \cap V(x_1, \dotsc, x_{\ell})$ has a universal Gr\"{o}bner basis consisting of squarefree-supported polynomials, obtained by evaluating the original universal Gr\"{o}bner basis at $x_1 = \dotsb = x_{\ell} = 0$. This implies that any Gr\"{o}bner degeneration of $Y \cap \mathbb{A}^n \cap V(x_1, \dotsc, x_{\ell})$ is reduced, so $Y \cap \mathbb{A}^n \cap V(x_1, \dotsc, x_{\ell})$ is reduced. 

Let $Z$ be an irreducible component of $Y \cap V(x_1, \dotsc, x_{\ell})$. Because $Y$ is integral, either $x_1$ vanishes on $Y$, or $Y \cap V(x_1)$ is pure of codimension $1$ in $Y$, and the fundamental class of $Y \cap V(x_1)$ is $[V(x_1)] \smallfrown [Y]$. In particular, in either case the fundamental class of $Y \cap V(x_1)$ is multiplicity-free. As $Y \cap V(x_1)$ is pure-dimensional, each irreducible component of $Y \cap V(x_1)$ is multiplicity-free. By choosing a component of $Y \cap V(x_1)$ which contains $Z$ and repeating this argument, we see that $Z$ is multiplicity-free. 
\end{proof}

We are interested in the singularities of multiplicity-free subvarieties. As properties of singularities are local, our results apply to varieties which locally look like multiplicity-free varieties. We say that a variety $X$ over $k$ is of \emph{multiplicity-free type} if $X$ is covered by open subsets which are isomorphic to open subsets of multiplicity-free subvarieties of products of projective lines. Theorem~\ref{thm:main} implies that a variety of multiplicity-free type is geometrically $F$-rational. 

We say that a subvariety $X$ of a product of projective spaces $\mathbb{P}^{a_1} \times \dotsb \times \mathbb{P}^{a_n}$ is multiplicity-free if the expansion of the fundamental class of $X$ in the basis of $A_*(\mathbb{P}^{a_1} \times \dotsb \times \mathbb{P}^{a_n})$ given by classes of coordinate subspaces has all coefficients equal to $0$ or $1$.
The proof of the following result is a modification of the proof of \cite[Proposition 1.5]{HL24}.

\begin{proposition}\label{prop:product}
Let $X$ be a multiplicity-free subvariety of $\mathbb{P}^{a_1} \times \dotsb \times \mathbb{P}^{a_n}$. Then $X$ is of multiplicity-free type. 
\end{proposition}

\begin{proof}
By extending scalars, we can reduce to the case when $k$ is infinite. 
Let $N = a_1 + \dotsb + a_n$, and let $U$ be a copy of $\mathbb{A}^{N}$ embedded in $\mathbb{P}^{a_1} \times \dotsb \times \mathbb{P}^{a_n}$ as the complement of the union of coordinate hyperplanes. 
By applying an automorphism of $\mathbb{P}^{a_1} \times \dotsb \times \mathbb{P}^{a_n}$, we may assume that $X \cap U$ is dense in $X$. Let $Y$ denote the closure of $U \cap X$ in $(\mathbb{P}^1)^N$.  It suffices to show that $Y$ is multiplicity-free. 

Let $S$ be a subset of $\{1, \dotsc, N\}$ of size equal to the dimension of $X$. Let $a_S$ be the degree of the cap product $[\bigcap_{i \in S} V(x_i)] \smallfrown [Y]$ in the Chow homology of $(\mathbb{P}^1)^N$. For each $i \in \{1, \dotsc, N\}$, let $\lambda_i$ be a general element of $k$. 
Let $X^{sm}$ be the smooth locus of $X$. 
By a version of Bertini's theorem \cite[Lemma B.9.1]{FultonIntersection}, $U \cap X^{sm} \cap \bigcap_{i \in S} V(x_i - \lambda_i)$ will be a $0$-dimensional subscheme of $U$ of length $a_S$. 

Let $V_j$ be the closure of $U \cap V(x_j - \lambda_j)$ in $\mathbb{P}^{a_1} \times \dotsb \times \mathbb{P}^{a_n}$. Because $X$ is multiplicity-free, the degree of $[\bigcap_{j \in S} V_j] \smallfrown [X]$ is either $0$ or $1$. Let $C_1, \dotsc, C_r$ be the irreducible components of $X \cap \bigcap_{j \in S} V_j$. 
By \cite[Section 6.1]{FultonIntersection}, we can write
\begin{equation}\label{eq:class}
[\bigcap_{j \in S} V_j] \smallfrown [X] = \sum_{i=1}^{r} m_i \alpha_i,
\end{equation}
where $\alpha_i$ is a class pushed forward from $C_i$ and $m_i$ is a positive integer. If $C_i$ has dimension $0$, then $\alpha_i$ is equal to the class of the reduction of $C_i$ and $m_i$ is bounded above by the multiplicity of $C_i$ \cite[Proposition 7.1]{FultonIntersection}. 
For each $C_i$ contained in $U \cap X^{sm}$, \cite[Proposition 7.1]{FultonIntersection} shows that $m_i$ is equal to the multiplicity of $C_i$. 
Because the tangent bundle of $\mathbb{P}^{a_1} \times \dotsb \times \mathbb{P}^{a_n}$ is globally generated, by \cite[Theorem 12.2(a)]{FultonIntersection}, each $\alpha_i$ is an effective class. When we take the degree of \eqref{eq:class}, the terms corresponding to the components which are contained in $U \cap X^{sm}$ contribute $a_S$ and all other terms contribute non-negatively, so $a_S \le 1$. 
\end{proof}

Given a subvariety $Y$ of $(\mathbb{P}^1)^n$, its \emph{multicone} is the subvariety of $\mathbb{A}^{2n}$ cut out by the ideal of $Y$ in the homogeneous coordinate ring of $(\mathbb{P}^1)^n$.

\begin{proposition}\label{prop:cone}
Let $Y$ be a multiplicity-free subvariety of $(\mathbb{P}^1)^n$. Then the multicone of $Y$ in $\mathbb{A}^{2n}$ is of multiplicity-free type. 
\end{proposition}

\begin{proof}
The closure of the multicone in $(\mathbb{P}^2)^n$ is a multiplicity-free subvariety: if $H_1, \dotsc, H_n$ are coordinate hyperplanes in $(\mathbb{P}^1)^n$ and $[Y] = \sum_{S} a_S \prod_{i \in S} [H_i] \smallfrown [(\mathbb{P}^1)^n]$, then the closure of the multicone in $(\mathbb{P}^2)^n$ has fundamental class given by $\sum_S a_S \prod_{i \in S} [H_i] \smallfrown [(\mathbb{P}^2)^n]$. The result then follows from Proposition~\ref{prop:product}. 
\end{proof}

\subsection{Examples}

We now use the tools developed above to give some interesting examples of varieties of multiplicity-free type. These examples show that multiplicity-free varieties need not have strongly $F$-regular or even $F$-pure singularities.

\begin{example}\label{ex:hypersurface}
Let $f \in k[x_1, \dotsc, x_n]$ be a squarefree-supported polynomial, i.e., all monomials appearing in $f$ with nonzero coefficient are squarefree. Assume that $f$ is irreducible. Then $V(f)$ is of multiplicity-free type, and the closure of $V(f)$ in $(\mathbb{P}^1)^n$ is multiplicity-free. Theorem~\ref{thm:main} then recovers the main result of \cite{Fsingsquarefree}, which is a strengthening of the main result of \cite{BWMW} and of \cite[Theorem 1.2]{BW}. 
\end{example}

In particular, Example~\ref{ex:hypersurface} includes configuration hypersurfaces, a class of polynomials introduced in \cite{BEK}. It was shown in \cite{BelkaleBrosnan} that the classes of a particular family of configuration hypersurfaces generate a localization of the Grothendieck ring of varieties. 

A direct proof that the hypersurfaces in Example~\ref{ex:hypersurface} have rational singularities was given in \cite{BWMW}. A stronger statement about singularities of pairs was proven in \cite{Sarkar}.

\begin{example}\label{ex:determinant}
Certain determinantal varieties are of multiplicity-free type. For example, for any natural numbers $m \le n$, the locus of matrices in $\mathbb{A}^{mn}$ which have either rank at most $1$, or which have rank at most $m-1$, is of multiplicity-free type; see \cite[Section 2.3]{HL24}. If $w$ is a permutation such that the Schubert polynomial has all coefficients equal to $0$ or $1$, then the matrix Schubert variety is of multiplicity-free type. See \cite[Corollary 1.6]{HL24}. 

If $1 < m < n$, then the locus of matrices of rank at most $1$, or of rank at most $m-1$, is not Gorenstein \cite[Corollary 8.9]{BrunsVetter} or even $\mathbb{Q}$-Gorenstein \cite[Corollary 8.4]{BrunsVetter}. In particular, varieties of multiplicity-free type need not be $\mathbb{Q}$-Gorenstein. 
\end{example}

It is known that matrix Schubert varieties are strongly $F$-regular, see \cite[Theorem 2.4.3]{KM05} and \cite{SchubertFregular}, which gives another proof that the varieties in Example~\ref{ex:determinant} are geometrically $F$-rational. 

\begin{example}
For some natural numbers $m \le n$, consider the subscheme of $\mathbb{A}^{mn}$ given by $m \times n$ matrices of rank at most $m-1$ where some subset of the entries are required to be equal to some fixed elements of $k$. For example, geometric properties of the subscheme of $\mathbb{A}^{mn}$ consisting of $m \times n$ matrices of rank at most $m - 1$ where some of the entries are $0$ were studied in depth in \cite{GiustiMerle} and its ideal was analyzed in \cite{BoocherDeterminantal}. Proposition~\ref{prop:slice} and Example~\ref{ex:determinant} show that this scheme is reduced. This scheme may be reducible, but each irreducible component is of multiplicity-free type. 
\end{example}

\begin{example}\label{ex:schubert}
Let $L \subseteq k^n$ be a linear subspace, and let $Y_L$ be the closure of $L$ in $(\mathbb{P}^1)^n$. This variety, which is known as the arrangement Schubert variety of $L$, was introduced in \cite{ArdilaBoocher}. By \cite[Theorem 1.3(c)]{ArdilaBoocher}, $Y_L$ is multiplicity-free. More generally, if $S_1, S_2, \dotsc, S_t$ are proper subspaces of $L$ with $\cap S_i = 0$, then the closure of the image of the map $\mathbb{P}L \dasharrow \mathbb{P}(L/S_1) \times \dotsb \times \mathbb{P}(L/S_t)$ is multiplicity-free \cite{BinglinLi}. For example, this includes the multiview varieties studied in \cite{AST}. 

Colin Crowley explained to the author that $Y_L$ is often not $\mathbb{Q}$-Gorenstein, and it can fail to have log-canonical singularities even if it is $\mathbb{Q}$-Gorenstein. Suppose that $L$ is a general $r$-dimensional subspace of $k^n$ with $1 < r < n$; this guarantees that $Y_L$ is $\mathbb{Q}$-Gorenstein. By \cite[Proposition 4.4]{GeoZonotopeII}, $K_{Y_L}$ is equal to $-\frac{2}{n - r + 1}$ times the divisor class of the restriction of $\mathcal{O}(1, \dotsc, 1)$ to $Y_L$. 
There is a distinguished resolution of $Y_L$ introduced in \cite{BHMPW20a} called the augmented wonderful variety. The smallest discrepancy of an exceptional divisor in this resolution is $\frac{2n}{n - r + 1} - r - 1$. In particular, $Y_L$ does not have log-canonical singularities if $n=7$ and $r=3$. 
\end{example}

\begin{example}
Let $L \subseteq k^n$ be a linear subspace that is not contained in any coordinate hyperplane. If $n > 1$, then for each $\{i, j\} \subseteq \{1, \dotsc, n\}$, there is a rational map $\mathbb{P}L \dasharrow \mathbb{P}^1$ induced by quotienting by the subspace $L \cap \{x_i = x_j = 0\}$. The closure of the image of the induced map $\mathbb{P}L \dasharrow (\mathbb{P}^1)^{\binom{n}{2}}$ is called the \emph{wonderful variety} $W_L$ of $L$. This is a smooth projective variety that can be constructed as an iterated blow-up of $\mathbb{P}L$ along strict transforms of subspaces; see \cite{dCP95}. It is a multiplicity-free subvariety of $(\mathbb{P}^1)^{\binom{n}{2}}$. 

There are examples of ample line bundles on wonderful varieties which have nonzero higher cohomology, see, e.g., \cite[Example 5.1]{EFL}. In particular, this gives an example where $W_L$ is not globally $F$-split; see \cite[Theorem 1.2.8]{BrionKumar}. By Proposition~\ref{prop:cone}, the multicone over $W_L$ in $\mathbb{A}^{2 \binom{n}{2}}$ is of multiplicity-free type. The multicone over $W_L$ is $F$-pure at the origin if and only if $W_L$ is globally $F$-split; see \cite[Section 2]{Hashimoto}. In particular, multiplicity-free varieties need not be $F$-pure. 
\end{example}

\begin{example}
Let $L \subseteq k^n$ be a linear subspace which is not contained in any coordinate hyperplane, and let $T$ be the standard torus embedded in $\mathbb{P}^{n-1}$. Let $P$ be a generalized permutohedron in $\mathbb{R}^n$, i.e., a lattice polytope whose edge directions are all parallel to $e_i - e_j$ for some $i$ and $j$. For convenience, we assume that $P$ has the maximal possible dimension, i.e., $\dim P = n-1$. Let $X_P$ be the corresponding toric variety, which contains $T$ as a dense open subset. It follows from \cite[Proposition 3.10]{EFL} that the closure of $\mathbb{P} L \cap T$ in $X_P$ is of multiplicity-free type. 
This includes the reciprocal plane of a linear space \cite{PS06} and Kapranov's visible contours compactification of a hyperplane arrangement complement \cite{Kapranov,HKTI}. 
\end{example}

\begin{example}\label{ex:bloch}
Let $L_1 \subseteq k^n$ and $L_2 \subseteq k^n$ be linear subspaces, and assume that $L_2$ is not contained in any coordinate hyperplane. Let $Z_{L_1, L_2}$ be the closure of $\{(v, u/v) : u \in L_1, \, v \in L_2 \cap \mathbb{G}_m^n\}$ in $\mathbb{A}^{2n}$; here the division is coordinate-wise. 

We claim that $Z_{L_1, L_2}$ is of multiplicity-free type. Indeed, the closure of $Z_{L_1, L_2}$ in $(\mathbb{P}^1)^{2n}$ is multiplicity-free. To check this, we may reduce to the case when $k$ is infinite. Let $S$ be a subset of $\{1, \dotsc, 2n\}$ of size equal to $\dim Z_{L_1, L_2}$. For each $i \in S$, let $\lambda_i$ be a general element of $k$.
To check that the closure of $Z_{L_1, L_2}$ in $(\mathbb{P}^1)^{2n}$ is multiplicity-free, it suffices to check that the intersection of $Z_{L_1, L_2}$ with $\cap_{i \in S} V(x_i - \lambda_i)$ is transverse and is either empty or consists of a reduced point.

Note that there is a birational map from $Z_{L_1, L_2}$ to $L_2 \times L_1$ given by $(x, y) \mapsto (x, xy)$, so the dimension of $Z_{L_1, L_2}$ is $\dim L_1 + \dim L_2$. By a version of Bertini's theorem \cite[B.9.1]{FultonIntersection}, we can arrange that the intersection of $Z_{L_1, L_2}$ with $\cap_{i \in S} V(x_i - \lambda_i)$ occurs inside the locus where this map is an isomorphism and that the intersection is dimensionally transverse.  If we pull back an equation of the form $x_i - \lambda_i$ to $L_2 \times L_1$ under this birational map, it becomes an affine equation in $L_2 \times L_1$. The intersection of $\dim L_1  + \dim L_2$-many dimensionally transverse affine equations in an open subset of $L_2 \times L_1$ is either empty or a reduced point. 

When $L_1 = L_2^{\perp}$, $Z_{L_1, L_2}$ is the multicone over a subvariety of $\mathbb{P}^{n-1} \times \mathbb{P}^{n-1}$ which was introduced in \cite[Section 4]{BEK} and studied in \cite{Bloch} and \cite{BDSW}; in \cite{BDSW} it is called Bloch's incidence variety. 
Theorem~\ref{thm:main} and Proposition~\ref{prop:cone} then recover the fact that the multicone over Bloch's incidence variety has $F$-rational singularities, first proved in \cite[Theorem 3.31]{BDSW}.
\end{example}

\begin{example}
Let $L_1 \subseteq k^n$ and $L_2 \subseteq k^n$ be linear subspaces, and assume that $L_2$ is not contained in any coordinate hyperplane. The semi-inverted Hadamard product of $L_1$ and $L_2$ is the closure of $\{u/v : u \in L_1, \, v \in L_2 \cap \mathbb{G}_m^n\}$ in $\mathbb{A}^n$; here the division is coordinate-wise. It follows from an argument analogous to the one in Example~\ref{ex:bloch} (or applying \cite[Proposition 2.9]{EFL} to Example~\ref{ex:bloch}) that semi-inverted Hadamard products are of multiplicity-free type. The fact that, over a field of characteristic $0$, semi-inverted Hadamard products have rational singularities plays an important role in \cite{BF24} and \cite{LiuVanishing}. 
\end{example}

\begin{example}
Choose some $r \le n$, and let $A \in \mathbb{A}^{rn}$ be a matrix of rank $r$. The group $GL_r \times \mathbb{G}_m^n$ acts on $\mathbb{A}^{rn}$, with $GL_r$ acting on the left by multiplication and $\mathbb{G}_m^n$ acting on the right by column scaling. Let $X_A$ be the closure of the orbit of $A$. 
Then the argument in \cite[Section 4]{BergetFink} and Proposition~\ref{prop:cone} show that $X_A$ is of multiplicity-free type. 
\end{example}

\section{Proof of Theorem~\ref{thm:main}}

Let $Y$ be a multiplicity-free subvariety of $(\mathbb{P}^1)^n$. 
By Proposition~\ref{prop:geomintegral}, $Y$ is geometrically integral, so the base change of $Y$ to the algebraic closure of $k$ is a multiplicity-free subvariety of $(\mathbb{P}^1)^n$. We may therefore assume that $k$ is algebraically closed. 

Recall that by Proposition~\ref{prop:CM}, $Y$ is normal and Cohen--Macaulay. For $i \in \{1, \dotsc, n\}$, let $p_i \colon (\mathbb{P}^1)^n \to (\mathbb{P}^1)^{n-1}$ be the coordinate projection that forgets the $i$th factor. Note that the fibers of the restriction of $p_i$ to $Y$ have dimension at most $1$, so for any coherent sheaf $\mathcal{F}$ on $Y$, we have $R^j p_{i*} \mathcal{F} = 0$ for $j > 1$. 

\begin{lemma}\label{lem:push}
For each $i$, we have $p_{i*} \mathcal{O}_Y = \mathcal{O}_{p_i(Y)}$ and $R^j p_{i*} \mathcal{O}_Y = 0$ for $j > 0$. 
\end{lemma}

Taking into account \cite[Theorem 1]{BrionMultiplicity}, Lemma~\ref{lem:push} is a special case of \cite[Proposition 2.9]{EFL}. We give a simple direct argument. 

\begin{proof}
By Proposition~\ref{prop:image}, $p_i(Y)$ is a multiplicity-free subvariety of $(\mathbb{P}^1)^{n-1}$, and so it is normal by Proposition~\ref{prop:CM}. If the restriction of $p_i$ to $Y$ is not birational, then the restriction is the projection $p_i(Y) \times \mathbb{P}^1 \to p_i(Y)$ by Proposition~\ref{prop:projectionproduct}, giving the result. Otherwise, the restriction of $p_i$ is birational, so Zariski's main theorem \cite[Exercise 28.4.B]{VakilNotes} implies that $p_{i*} \mathcal{O}_Y = \mathcal{O}_{p_i(Y)}$. Let $\mathcal{I}$ be the ideal sheaf of $Y$ in $(\mathbb{P}^1)^n$, so we have an exact sequence $0 \to \mathcal{I} \to \mathcal{O}_{(\mathbb{P}^1)^{n}} \to \mathcal{O}_Y \to 0$. The long exact sequence induced by $p_{i*}$ looks like
$$\dotsb \to R^1 p_{i*} \mathcal{I} \to R^1p_{i*} \mathcal{O}_{(\mathbb{P}^1)^n} \to R^1p_{i*} \mathcal{O}_Y \to R^2p_{i*} \mathcal{I} \to \dotsb.$$
As $R^1p_{i*} \mathcal{O}_{(\mathbb{P}^1)^n} = R^2p_{i*} \mathcal{I} = 0$, this implies that $R^1p_{i*} \mathcal{O}_Y = 0$. 
\end{proof}

Recall that $Y$ and $p_i(Y)$ are Cohen--Macaulay by Proposition~\ref{prop:CM}. Applying Grothendieck duality to Lemma~\ref{lem:push} gives the following result. 

\begin{corollary}\label{prop:pushomega}
For any $i \in \{1, \dotsc, n\}$ such that $p_i$ is birational, we have $Rp_{i*} \omega_Y = \omega_{p_i(Y)}$. 
\end{corollary}

\begin{lemma}\label{lem:pushvanish}
Assume that $Y$ is not equal to $(\mathbb{P}^1)^n$. Let $\mathcal{F}$ be a coherent sheaf on $Y$, and assume that $p_{i*} \mathcal{F} = 0$ for all $i \in \{1, \dotsc, n\}$. Then $\mathcal{F} = 0$. 
\end{lemma}

\begin{proof}
Suppose that $\mathcal{F}$ is nonzero, and let $W$ be an irreducible component of the support of $\mathcal{F}$. Let $\mathcal{F}'$ be the subsheaf of $\mathcal{F}$ consisting of sections which are supported on $W$. It suffices to show that $\mathcal{F}'$ is zero. 

Because $W \not= (\mathbb{P}^1)^n$, there is some $i$ such that the restriction of $p_i$ to $W$ is generically finite. This implies that the support of $p_{i*} \mathcal{F}'$ is equal to $p_i(W)$. But the assumption that $p_{i*} \mathcal{F} = 0$ implies that $p_{i*} \mathcal{F}' = 0$, so $\mathcal{F}' = 0$. 
\end{proof}

Let $Y$ be an (integral) variety over an algebraically closed field of positive characteristic. An \emph{alteration} of $Y$ is a variety $X$ with a proper surjective generically finite map $X \to Y$. 
Recall that there is a subsheaf $\tau(\omega_Y)$ of $\omega_Y$ called the \emph{test module} with the property that $Y$ has $F$-rational singularities if and only if $\tau(\omega_Y) = \omega_Y$. This module, which was introduced in \cite{SmithTestIdeals}, is defined as the smallest nonzero subsheaf $\mathcal{F}$ of $\omega_Y$ such that the image of $\mathcal{F}$ under the trace of Frobenius is contained in $\mathcal{F}$. 
For any alteration $\pi \colon X \to Y$, the image of $\pi_* \omega_X$ in $\omega_Y$ satisfies this property, so the test module is contained in the image of $\pi_* \omega_X$. See \cite[Section 2.5]{BST} for a discussion of test modules. 

By \cite[Theorem 3.2]{BST}, there is an alteration $\pi \colon X \to Y$ such that $\tau(\omega_Y)$ is the image of $\pi_*\omega_X$ in $\omega_Y$. It follows that this also holds for any alteration which factors through $X$.

\begin{proof}[Proof of Theorem~\ref{thm:main}]
By induction, we may assume that $p_i(Y)$ has $F$-rational singularities for any $i \in \{1, \dotsc, n\}$. If some map $p_i$ is not birational, then $Y$ is isomorphic to $\mathbb{P}^1 \times p_i(Y)$ by Proposition~\ref{prop:projectionproduct}, so we may assume that $p_i$ is birational for each $i$. 

Choose an alteration $\pi \colon X \to Y$ such that $\tau(\omega_Y)$ is the image of $\pi_* \omega_X$ in $\omega_Y$. Let $K$ be the kernel of the map $\pi_* \omega_X \to \omega_Y$, and let $C = \omega_Y/\tau(\omega_Y)$. We need to show that $C = 0$. By Lemma~\ref{lem:pushvanish}, it suffices to show that $p_{i*}C = 0$ for each $i$. We have short exact sequences
\begin{equation}\label{eq:cokernel}
0 \to \tau(\omega_Y) \to \omega_Y \to C \to 0, \text{ and }
\end{equation}
\begin{equation}\label{eq:kernel}
0 \to K \to \pi_* \omega_X \to \tau(\omega_Y) \to 0.
\end{equation}
Set $p = p_i$ for some $i$, and let $Z = p(Y)$. Because the fibers of $p$ have dimension at most $1$, $R^jp_*$ vanishes for $j \ge 2$. By Corollary~\ref{prop:pushomega}, $p_* \omega_Y = \omega_Z$ and $R^1p_* \omega_Y = 0$. 
This implies that $p_* \tau(\omega_Y)$, thought of as a subsheaf of $\omega_Z$, contains the test module $\tau(\omega_Z)$. It follows that $p_* \tau(\omega_Y)$ is equal to $\omega_Z$ because $Z$ has $F$-rational singularities. Applying $p_*$ to \eqref{eq:cokernel}, we have an exact sequence
$$0 \to \omega_Z \xrightarrow{\sim} \omega_Z  \xrightarrow{0} p_*C \to R^1p_* \tau(\omega_Y) \to 0.$$
To prove the theorem, it therefore suffices to show that $R^1p_* \tau(\omega_Y) = 0$. Pushing forward \eqref{eq:kernel}, we have an exact sequence
$$0 \to p_*K \to p_* \pi_* \omega_X \to \omega_Z \to R^1p_* K \to R^1p_* \pi_* \omega_X \to R^1 p_* \tau(\omega_Y) \to 0.$$
By \cite[Lemma 5.2]{BST}, we can find an alteration $\rho \colon U \to X$ such that the map $R^1 (p \circ \pi \circ \rho)_* \omega_U \to R^1 (p \circ \pi)_* \omega_X$ is $0$. We apply the Leray spectral sequence to the composition of $\pi \circ \rho$ with $p$. Because $R^j p_*$ vanishes for $j > 1$, the $E_2$ page of the associated Leray spectral sequence looks like the following:
\begin{center}
\begin{tikzpicture}
  \matrix (m) [matrix of math nodes,
    nodes in empty cells,nodes={minimum width=5ex,
    minimum height=20,outer sep=2pt},
    column sep=1,row sep=1pt]{
          \vdots     &  \vdots &  \vdots  & \hdots & \vdots & \\
          2      & p_{*} R^2 (\pi \circ \rho)_* \omega_U                  & R^1 p_{*} R^2 (\pi \circ \rho)_* \omega_U &   0             & \hdots & \\
          1      & p_{*} R^1 (\pi \circ \rho)_* \omega_U                  & R^1 p_{*} R^1 (\pi \circ \rho)_* \omega_U &   0             & \hdots & \\
          0     &  p_{*} (\pi \circ \rho)_* \omega_U  & R^1p_{*} (\pi \circ \rho)_* \omega_U &  0  & \hdots & \\
    \quad\strut &   0  &  1  &  2  & \hdots  &\strut \\};
\draw[thick] (m-1-1.east) -- (m-5-1.east) ;
\draw[thick] (m-5-1.north) -- (m-5-6.north) ;
\end{tikzpicture}
\end{center}
This spectral sequence degenerates at $E_2$ because it is supported in two columns, so the edge map $R^1p_* (\pi \circ \rho)_* \omega_U \to R^1(p \circ \pi \circ \rho)_* \omega_U$ is injective. Similarly, the map $R^1 p_* \pi_* \omega_X \to R^1(p \circ \pi)_* \omega_X$ is injective. It follows that the map $R^1p_* (\pi \circ \rho)_* \omega_U \to R^1p_* \pi_* \omega_X$ is $0$. 

Because the map $U \to Y$ factors through $X$, the image of $(\pi \circ \rho)_* \omega_U$ in $\omega_Y$ is $\tau(\omega_Y)$. Replacing $X$ by $U$ and pushing forward the sequence analogous to \eqref{eq:kernel}, we see that the map $R^1p_* (\pi \circ \rho)_* \omega_U \to R^1 p_* \tau(\omega_Y)$ is surjective. But this map factors through $R^1p_* \pi_* \omega_X$, so the map is $0$, proving the result. 
\end{proof}

\begin{remark}\label{rem:normal}
The argument used to prove Theorem~\ref{thm:main} can easily be adapted to prove that multiplicity-free varieties are normal and, if $k$ has characteristic $0$, have rational singularities. For example, if $\nu \colon Y^{\nu} \to Y$ is the normalization of $Y$ and $p(Y)$ is inductively assumed to be normal, then applying $p_*$ to the sequence $0 \to \mathcal{O}_{Y} \to \nu_* \mathcal{O}_{Y^{\nu}} \to \nu_* \mathcal{O}_{Y^{\nu}}/\mathcal{O}_Y \to 0$ and arguing as in the proof of Theorem~\ref{thm:main} shows that $\nu_* \mathcal{O}_{Y^{\nu}}/\mathcal{O}_Y  = 0$. 
\end{remark}

\begin{remark}\label{rem:rational}
Several of the above arguments are analogous to arguments that Brion uses to prove Proposition~\ref{prop:CM}, although the argument is set up in a different way. For example, Lemma~\ref{lem:push} is analogous to \cite[Lemma 3]{BrionMultiplicity} and Lemma~\ref{lem:pushvanish} is analogous to \cite[Lemma 8]{BrionOrbitClosure}. 
\end{remark}

\subsection{Multiplicity-free subvarieties of a flag variety}\label{ssec:GmodP}
Let $G$ be a reductive group, let $P$ be a reduced parabolic in $G$, and let $B$ be a Borel subgroup contained in $P$. Recall that the Chow homology of $G/P$ has a basis consisting of classes of Schubert varieties, and that a subvariety $Y$ of $G/P$ is \emph{multiplicity-free} if, when the fundamental class of $Y$ is expressed in this basis, all coefficients are equal to $0$ or $1$. In \cite{BrionOrbitClosure,BrionMultiplicity}, Brion showed that multiplicity-free subvarieties of $G/P$ are normal and Cohen--Macaulay, and, if $k$ has characteristic $0$, they have rational singularities. We now sketch a proof of the following strengthening of Theorem~\ref{thm:main}.

\begin{theorem}\label{thm:GmodP}
Let $Y$ be a multiplicity-free integral subvariety of $G/P$ over a field $k$ of positive characteristic. Then $Y$ has geometrically $F$-rational singularities. 
\end{theorem}

We reiterate that most examples that we are aware of multiplicity-free subvarieties of a $G/P$ which is not a product of projective spaces are known to satisfy the stronger property of being strongly $F$-regular. 

\begin{proof}
As in the proof of Theorem~\ref{thm:main}, we may assume that $k$ is algebraically closed. 
The map $G/B \to G/P$ is smooth with irreducible fibers, so the preimage of $Y$ in $G/B$ is irreducible. The preimage of a Schubert variety in $G/P$ is a Schubert variety in $G/B$, so the preimage of $Y$ in $G/B$ is multiplicity-free. As the preimage of $Y$ in $G/B$ is smooth over $Y$, it suffices to prove that the preimage of $Y$ in $G/B$ has $F$-rational singularities; see \cite[Proposition A.5]{DattMurayama} and \cite[(6) on pg. 440]{Velez}.  We may therefore assume that $P = B$. 

Let $P_i$ be a minimal parabolic properly containing $B$. The map $G/B \to G/P_i$ is a $\mathbb{P}^1$-bundle, and we have a Cartesian diagram
\begin{center}
\begin{tikzcd}
G/B \times_{G/P_i} G/B \arrow[r, "p_i"] \arrow[d, "p_i"]
& G/B \arrow[d] \\
G/B \arrow[r]
&G/P_i 
\end{tikzcd}.
\end{center}
Note that $p_i(p_i^{-1}Y)$ is equal to $P_i \cdot Y$, the orbit of $Y$ under $P_i$. We have $\dim P_i \cdot Y \in \{\dim Y, \dim Y + 1\}$, and we say that $P_i$ \emph{raises} $Y$ if $\dim P_i \cdot Y = \dim Y + 1$. If $P_i$ raises $Y$, then the homology class of $P_i \cdot Y$ is equal to $p_{i*} p_i^* [Y]/d$, where $d$ is the degree of the generically finite map $p_i^{-1}(Y) \to P_i \cdot Y$. 

If $W$ is a Schubert variety in $G/B$, then $P_i \cdot W$ is also a Schubert variety. If $V$ is a different Schubert variety of the same dimension, then $P_i \cdot V$ is different from $P_i \cdot W$ as long as $P_i$ raises both $V$ and $W$. This implies that if $Y$ is multiplicity-free, then so is $P_i \cdot Y$, so the map $p_i \colon p_i^{-1}(Y) \to P_i \cdot Y$ is birational if $P_i$ raises $Y$. The fibers of this map are at most $1$-dimensional. By an argument similar to the proof of Lemma~\ref{lem:push}, using that $P_i \cdot Y$ is normal by \cite{BrionMultiplicity}, we have $p_{i*} \mathcal{O}_{p_i^{-1}(Y)} = \mathcal{O}_{P_i \cdot Y}$ and $R^j p_{i*} \mathcal{O}_{p_i^{-1}(Y)} = 0$ for $j > 0$. 

For any nonempty $X \subsetneq G/B$, there is some minimal parabolic $P_i$ which raises $X$. Indeed, it suffices to check this for Schubert varieties, as whether $P_i$ raises $X$ can be checked in terms of the fundamental class of $X$. 

By induction on the codimension, we may assume that for each parabolic $P_i$ that raises $Y$, $P_i \cdot Y$ has $F$-rational singularities. Let $\tau(\omega_Y)$ be the test module of $Y$. If $Y$ is not $F$-rational, then choose an irreducible component $U$ of the support of $\omega_Y/\tau(\omega_Y)$. If a minimal parabolic does not raise $Y$, then the support of $\omega_Y/\tau(\omega_Y)$ is invariant under the action of that parabolic, so it does not raise $U$. Therefore, any minimal parabolic which raises $U$ also raises $Y$. Choose such a minimal parabolic $P_i$. 

As $p_i^{-1}(Y)$ is a $\mathbb{P}^1$-bundle over $Y$, the support of $\omega_{p_i^{-1}(Y)}/\tau(\omega_{p_i^{-1}(Y)})$ is the inverse image of the support of $\omega_Y/\tau(\omega_Y)$. In particular, $p_i^{-1}(U)$ is an irreducible component of the support of $\omega_{p_i^{-1}(Y)}/\tau(\omega_{p_i^{-1}(Y)})$. As the map $p_i^{-1}(U) \to P_i \cdot Y$ is generically finite onto its image, this implies that $p_{i*}\omega_{p_i^{-1}(Y)}/\tau(\omega_{p_i^{-1}(Y)})$ is nonzero. 

To obtain a contradiction, we run the rest of the argument in the proof of Theorem~\ref{thm:main}. In brief, choose an alteration $\pi \colon X \to p_i^{-1}(Y)$ which has the property that $\tau(\omega_{p_i^{-1}(Y)})$ is the image of $\pi_* \omega_X$ in $\omega_{p_i^{-1}(Y)}$. Pushing forward the analogues of \eqref{eq:cokernel} and \eqref{eq:kernel} show that $p_{i*} \omega_{p_i^{-1}(Y)}/\tau(\omega_{p_i^{-1}(Y)})$ is isomorphic to $R^1 p_{i*} \tau(\omega_{p_i^{-1}(Y)})$, and that the map $R^1 p_{i*} \pi_* \omega_X \to R^1 p_{i*} \tau(\omega_{p_i^{-1}(Y)})$ is surjective. But by choosing a further alteration and using \cite[Lemma 5.2]{BST}, we can show that this map is $0$. 
\end{proof}

\bibstyle{alpha}
\bibliography{matroid.bib}

\end{document}